\documentclass[a4paper]{elsarticle}
\usepackage[T1]{fontenc}
\usepackage[utf8]{inputenc}
\usepackage[english]{babel}
\usepackage{amsmath}
\usepackage{amssymb}
\usepackage{amsthm}
\usepackage{pictexwd,dcpic}
\usepackage{tikz-cd}

\theoremstyle{plain} 
\newtheorem{thm}{Theorem} 
\newtheorem{cor}[thm]{Corollary} 
\newtheorem{lem}[thm]{Lemma}

\newtheorem{example}[thm]{Example}

\newcommand{\Hom}{\mathit{Hom}}
\newcommand{\Ker}{\mathit{Ker}}
\newcommand{\Id}{\mathit{Id}}

\begin{document}

\title{Injective and projective semimodules over involutive semirings}
\date{July 21, 2020}

\author[1]{Peter Jipsen} 
\ead{jipsen@chapman.edu}
\address[1]{Chapman University, Faculty of Mathematics, California, USA.}

\author[2]{Sara Vannucci}
\ead{svannucci@unisa.it}
\address[2]{Dipartimento di Matematica,  Universit\`{a} di Salerno, Fisciano, Italy.}

\begin{abstract}
We show that the term equivalence between MV-algebras and MV-semirings lifts to involutive residuated lattices and a class of semirings called \textit{involutive semirings}. The semiring perspective helps us find a necessary and sufficient condition for the interval $[0,1]$ to be a subalgebra of an involutive residuated lattice. We also import some results and techniques of semimodule theory in the study of this class of semirings, generalizing results about injective and projective MV-semimodules. Indeed, we note that the involution plays a crucial role and that the results for MV-semirings are still true for involutive semirings whenever the Mundici functor is not involved. In particular, we prove that involution is a necessary and sufficient condition in order for projective and injective semimodules to coincide.
\end{abstract}


\maketitle

\section{Introduction}
Semirings and semimodules, and their applications, arise in various branches of mathematics, computer science, physics, as well as in many other areas of modern science (see, for example \cite{golan:sata}). Involutive residuated lattices arose in the literature as generalizations of classical propositional logic and classical linear logic (\cite{GJKO}). The article is organized as follows. In Section 1, for the reader's convenience, we provide all the necessary notions on semirings and involutive residuated lattices. Then, we prove a term equivalence between involutive residuated lattices and a special class of semirings that we shall call involutive $0$-free semirings. This categorical isomorphism helps us to find a necessary and sufficient condition for $[0,1]$ to be a subalgebra of an involutive residuated lattice. In particular, we show that $[0,1]$ is a subalgebra of an involutive residuated lattice if and only if $0$ is a multiplicatively idempotent element.
In the second section we focus our attention on those involutive residuated lattices for which $0$ is the bottom element. In this case $1$ is the top and so the lattice is bounded. We consider involutive semiring and we define involutive semimodules. Then we characterize injective and projective involutive semimodules, generalizing similar result for MV-semirings in \cite{dlnv:isaidscmv}. Indeed, involution seems to play a crucial role and the results for MV-semirings seem to be generalizable whenever the Mundici functor is not involved. We show for example that, for a finite commutative involutive semiring, injective and projective finitely generated semimodules coincide and we show, by providing a counterexample, that without the involution this is not true. 
It leads us to observe that, even if the involution appears only in the semiring and it doesn't affect at all the structure of the semimodule, it still plays a fundamental role in the study of injective and projective semimodules. 
Furthermore, we restate a well-known characterization of injective semimodules over a semiring $A$ in terms of the join-semilattice $Id(A)$ (the ideals of $A$ considered as a join-semilattice) with the reverse order.

\section{Involutive Semirings}

A \emph{0-free} \emph{semiring} is an algebra $(A, +, \cdot, 1)$ such that
\begin{itemize}
\item $(A, +)$ is a commutative semigroup,
\item $(A, \cdot, 1)$ is a monoid and
\item $a (b+c)=a b + a c$ and $(a+b) c=a c + b c$ for all $a, b, c \in A$
\end{itemize}
where, as usual, $ab$ is short for $a\cdot b$. A \emph{semiring} $(A,+,0,\cdot,1)$ satisfies the same axioms, as well as $a+0=a$ and $a0=0=0a$ for all $a\in A$. A \emph{semifield} is a semiring in which all non-zero elements have a multiplicative inverse.

A (0-free) semiring $A$ is \emph{commutative} if $a \cdot b = b \cdot a$, for all $a, b \in A$, it is \emph{1-bounded} if $a+1=1$
and (\emph{additively}) \emph{idempotent} if $a+a=a$ for every $a \in A$ (or equivalently $1+1=1$). 
Note that an idempotent semiring has a natural order defined on it by $x \leq y \iff x+y=y$, in which case $+$ is denoted by $\vee$ and $(S,\vee)$ is a join-semilattice. In the 1-bounded case, the identity $1$ is the top element. Such
algebras are also called \emph{integral}, but we do not use this terminology here since integrality has a different meaning in ring theory.

A \emph{residuated join-semilattice} or \emph{0-free residuated idempotent semiring} is an algebra $(A, \vee, \cdot, 1, \backslash, /)$ such that
\begin{itemize}
\item $(A, \vee)$ is a semilattice, with partial order defined by $x\le y\iff x\vee y=y$,
\item $(A, \cdot, 1)$ is a monoid and
\item \textup{(res)} \quad $xy \le z \iff x \le z/y \iff y \le x\backslash z$
\quad holds for all $x,y,z\in A$.
\end{itemize}
A \emph{pointed residuated join-semilattice} $(A,\vee,\wedge,\cdot,1,\backslash,/,0)$ is a residuated lattice with an additional constant $0$. Note that this constant need not be the least element of the lattice. An \emph{involutive residuated lattice} is a pointed residuated join-semilattice that satisfies
${\sim}{-}x = x = -{\sim}x$ for all $x\in A$, where ${\sim}x = x\backslash 0$ and $-x=0/x$.

The operations $\sim,-$ are order-reversing, and are called \emph{left} and \emph{right linear negation}. The residuation equivalences (res) can be replaced by four identities, hence involutive residuated lattices form a variety, denoted by \textsf{InRL}. 
It is well known (and easy to see) that
$\backslash,/,0$ can be expressed by the linear negations and the monoid operation: 
$$x\backslash y = {\sim}((-y)x), \quad
x/y = -(y({\sim}x) \ \text{ and } \ 0 = {\sim}1 = -1.$$
Residuation also implies that $\cdot$ distributes over $\vee$, hence $(A,\vee,\cdot,1)$ is a $0$-free idempotent semiring. 

Note that the constant $0$ is an additive identity if and only if $0$ is the bottom element, or equivalently $1$ is the top element, i.\,e., the semiring is 
$1$-bounded. It follows from (res) that $x0=0=0x$, so a $1$-bounded 
involutive residuated lattice has a semiring reduct.
If $\cdot$ is commutative then $x\backslash y=y/x$, hence ${\sim}x=-x$.

An \emph{MV-algebra} is a $1$-bounded commutative involutive residuated 
lattice that satisfies $x\vee y=(x/y)\backslash x$, 
(though these algebras are usually defined 
using the operations $\oplus,-,0$, where $x\oplus y = {\sim}(-y\cdot -x)$) \cite{cdm:afomvr}.

An \emph{MV-semiring} \cite{dr:sasiimva} is an algebra $(A,\vee,0,\cdot,1,-)$ such
that
\begin{itemize}
\item $(A,\vee,0,\cdot,1)$ is a commutative idempotent semiring,
\item $x \le y \iff x \cdot -y = 0$ and
\item $x\vee y=-(-x\cdot -(-x\cdot y))$ for all $x,y\in A$.
\end{itemize}

Proposition 4.11 in \cite{dr:sasiimva} shows that MV-algebras and MV-semirings are
term-equivalent. This result has led to fruitful interaction between
research in fuzzy logic and semirings/semimodules \cite{dlnv:isaidscmv,dr:sasiimva}. Below we show that the term-equivalence lifts to involutive residuated lattices and a class of
non-commutative 0-free semirings. This expands the applicability of semiring
techniques to involutive residuated lattices since 1-boundedness and the third axiom of MV-semirings are not required to hold. The general term-equivalence was first shown between coupled semirings and involutive residuated lattices in \cite{jip:raisgbia}.
The variety of involutive residuated lattices is considerably more general than the variety of MV-algebras since the latter have distributive lattice reducts, while there are non-distributive 1-bounded involutive residuated lattices (the smallest examples have seven elements).

An \emph{involutive semiring} is an algebra $(A, \vee, \cdot, 1, {\sim}, -)$ such that
\begin{itemize}
\item $(A, \vee, \cdot, 1)$ is a $0$-free idempotent semiring and
\item $x \le y \iff x \cdot {\sim}y \le -1 \iff -y\cdot x\le -1$ for all $x,y\in A$.
\end{itemize}
The element $-1$ is denoted by $0$, although it need not be the bottom
element of the join-semilattice.

\begin{thm}\label{termeq}
Involutive residuated lattices are term-equivalent to involutive semirings.
\end{thm}

\begin{proof}
As mentioned before, involutive residuated lattices have 0-free semiring
reducts, and $x\le y$ is equivalent to $x\le -{\sim}y=0/{\sim}y$, which 
by (res) is equivalent to $x\cdot {\sim}y\le 0=-1$. The equivalence
$x\le y\iff -y\cdot x\le -1$ is proved similarly, showing that any involutive
residuated lattice is an involutive semiring.

Conversely, let $A$ be an involutive semiring and define $x\wedge y=
{\sim}(-x\vee -y)$. It remains to prove the identities
${\sim}{-}x=x=-{\sim}x$, that $(A,\vee,\wedge)$ is a lattice
and that (res) holds.

To prove the identity ${\sim}{-}x=x$, note that $-x\leq y \iff -x\cdot {\sim}y\leq 0 \iff {\sim}y\leq x$. Substituting $y$ by $-x$ we get ${\sim}{-}x\leq x$, hence ${\sim}{-}{\sim}{-}x\leq {\sim}{-}x\leq x$, or equivalently $-x\leq -{\sim}{-}x$.
Similarly we can substitute $x$ by ${\sim}y$ obtaining $-{\sim}y\leq y$, and replacing $y$ by $-x$ we have $-{\sim}{-}x\leq -x$, hence the identity $-x=-{\sim}{-}x$ holds.

Now $x\leq x$ implies $-x \cdot x\leq 0$, hence $-0 \cdot (-x \cdot x)\leq 0$. From the preceding identity it follows that $-{\sim}{-}0 \cdot (-{\sim}{-}x \cdot x)\leq 0$, which implies ${-}{\sim}{-}x \cdot x\leq {\sim}{-}0\leq 0$ and therefore $x\leq {\sim}{-}x$. So we have shown that the identity ${\sim}{-}x=x$ holds, and $-{\sim}x=x$ is proved similarly.

Next, observe that $x\le y\iff -y\cdot x\le 0\iff -y\cdot {\sim}-x\le 0\iff
-y\le -x$. A similar calculation for $\sim$ shows that the unary operations
in an involutive semiring are order-reversing inverses of each other. 
Since $(A,\vee)$ is a join-semilattice, it follows that $(A,\wedge)$ 
is a meet-semilattice with respect to the same order.

We now prove the absorption laws:
$x \wedge (x \vee y)= {\sim} (-x \vee -(x \vee y))={\sim}{-}x = x$ since $x \leq x \vee y$ implies $-(x \vee y) \leq -x$). Similarly
$x \vee (x \wedge y) = x\vee {\sim} (-x \vee -y) = x$ since $-x \leq -x \vee -y$ implies ${\sim}(-x \vee -y) \leq x$. Hence $(A,\vee,\wedge)$ is a lattice.

Finally we prove that if the residuals are defined as $x \backslash z = {\sim} (-z \cdot x)$ and $z/y = -(y \cdot {\sim} z)$ then (res) holds:
$$y \leq {\sim} (-z \cdot x) \Leftrightarrow -z \cdot x \leq -y \Leftrightarrow -z \cdot x \cdot y \leq -1 \Leftrightarrow -z \leq -(x \cdot y) \Leftrightarrow xy \leq z.$$
The second equivalence of (res) is proved similarly, hence any involutive
semiring determines an involutive residuated lattice. The term-equivalence is established by observing that $x\backslash 0={\sim}(-0\cdot x)={\sim}(-{\sim}1\cdot x)={\sim}x$ and likewise $0/x=-x$.
\end{proof}

In the next result we use standard interval notation, so $[0,1] = \{ a \mid 0 \leq a \leq 1\}$.

\begin{thm}
In any involutive semiring (equivalently involutive residuated lattice) the interval $[0,1]$ is a subalgebra if and only if $0$ is a multiplicative idempotent element, i.\,e., $0\cdot 0=0$.
\end{thm}

\begin{proof}
Assume $[0,1]$ is a subalgebra of an involutive semiring. Then $0\le 1$ since
any subalgebra contains the constant $1$. We also have that $0 \leq 1 \iff 0 \leq -{\sim}1 \iff 0 \cdot 0 \leq 0$. The reverse inequality $0\le 0\cdot 0$
holds because any subalgebra is closed under $\cdot$ and hence we obtain $0 \cdot 0 = 0$.

Conversely, assume $0\cdot 0=0$. The equivalence $0 \cdot 0 \leq 0 \iff 0 \leq 1$ shows that $0,1$ are in the interval $[0,1]$. The interval is certainly closed under joins, and closure under ${\sim}$ follows from
$$
0\le a\le 1\quad\iff\quad 0={\sim}1\le{\sim}a\le{\sim}0=1.
$$
The argument for closure under $-$ is the same.
As regards the multiplication we have that if $a, b \in [0,1]$ then $ab \leq a,b$ and in particular $ab \leq 1$.
Observe that if $a \leq b$ and $c \leq d$, we have that $ac \leq bd$, so $0 \leq a, b$ implies $0 \cdot 0 \leq ab$. Since we assume $0 \cdot 0 = 0$ we have that $0 \leq ab$, hence $a,b\in [0,1]$.
\end{proof} 

In an involutive residuated lattice $0$ is the bottom element if and only if $1$ is the top element. 
Hence \emph{1-bounded involutive semirings} can be defined as idempotent semirings with two unary operations ${\sim},-$ such that 
$$
x \le y \iff x \cdot {\sim}y =0 \iff -y\cdot x=0.
$$

\section{Semimodules over 1-bounded involutive semirings}

Let $A$ be a semiring. A (left) $A$-\emph{semimodule} is a commutative monoid 
$(M, +, 0)$ with a scalar multiplication $\cdot : A \times M \to M$, such that the following conditions hold for all $a, b \in A$ and $x, y \in M$:

\begin{itemize}
\item $(ab) \cdot x= a \cdot (b \cdot x)$
\item $a \cdot (x + y) = (a \cdot x) + (a \cdot y)$
\item $(a + b) \cdot x = (a \cdot x) + (b \cdot x)$
\item $0_{A} \cdot x= 0_{M}= a \cdot 0_{M}$
\item $1 \cdot x=x$.
\end{itemize}

For example, any semiring $A$ can be considered a \emph{regular} left $A$-semimodule with scalar multiplication $a\cdot x=ax$ for all $a,x\in A$.
The definition of right $A$-semimodules is completely analogous. From now on, we will refer generically to semimodules without specifying left or right and we will use the notations of left semimodules.

Since the intersection of $A$-semimodules is an $A$-semimodule we can define finitely generated and cyclic semimodules in a standard way, in particular an $A$-semimodule $M$ is cyclic if and only if there exists $m \in M$ such that $M = Am = \{am \mid a \in A\}$.
An example of left cyclic semimodule over a semiring $A$ is given by $Ax = \{ax \mid a \in A\}$ with $x \in A$, it is the principal left-ideal of the semiring $A$ generated by $x$. 
If $A$ is additively idempotent, then any $A$-semimodule is also
additively idempotent, hence a join-semilattice with $0$  (since $x=1x=(1+1)x=x+x$). In this case 
we write $\vee$ instead of $+$ and often make use of the natural order given
by $x\le y \iff x\vee y=y$.

Let $(M, +, 0)$ and $(N, +, 0)$ be two semimodules over a semiring $A$. For any subsemiring $B$ of $A$, we can consider $M,N$ semimodules over $B$. A \emph{$B$-semimodule homomorphism} is a function $f: M \to N$ such that $f(m + m')=f(m) + f(m')$ and $f(b \cdot m)=b \cdot f(m)$ for all $m, m' \in M$ and $b \in B$. The set of all such homomorphisms is denoted by $\Hom_B(M,N)$. If we take $M$ be the semiring $A$, considered as a semimodule over itself, then $\Hom_B(A,N)$ is an $A$-semimodule with pointwise addition, and scalar multiplication $a\cdot f$ given by $(a\cdot f)(t)=f(ta)$ for all $t\in A$. Note that $a$ has to act from the right if $A$ is noncommutative, while for commutative $A$ it holds in general that $\Hom_B(M,N)$ is an $A$-semimodule.

Let $A$ be a semiring and $(M, +, 0)$ an $A$-semimodule. 
An $A$-semimodule $E$ is \textit{injective} if and only if, given an $A$-semimodule $M$ and a subsemimodule $N$ of $M$, any semimodule homomorphism $\alpha$ from $N$ to $E$ can be extended to a semimodule homomorphism $\beta$ from $M$ to $E$ such that $\beta\iota=\alpha$. 
\medskip
\begin{center}
\begin{tikzcd}
N \arrow [d, hook, "\iota"] \arrow[r, "\alpha"] & E\\
M \arrow [ur,dashed,"\beta"] 
\end{tikzcd}
\end{center}
\medskip
A semiring $A$ is called \textit{self-injective} if the regular $A$-semimodule $A$ is injective.

Recall that $\mathbb B$ denotes the 2-element Boolean semifield. A semimodule $M$ is a \emph{retract} of a semimodule $M'$ if there exist homomorphisms $r:M'\to M,s:M\to M'$ such that the composition $rs$ is the identity map on $M$.
\begin{thm}\label{inj}  \textup{(}\cite{dlnv:isaidscmv}\textup{)}
Let $A$ be an additively idempotent semiring and $M$ an $A$-semimodule. Then $M$ is injective if and only if there exists a set $X$ such that $M$ is a retract of the $A$-semimodule $\Hom_{\mathbb{B}} (A, \mathbb{B})^{X}$.
\end{thm} 

Let $A$ be a semiring. An $A$-semimodule $P$ is \textit{projective} if  the following condition holds: if $\varphi: M \longrightarrow N$ is a surjective $A$-homomorphism of $A$-semimodules and if
$\alpha: P \longrightarrow N$ is an $A$-homomorphism then there exists an $A$-homomorphism $\beta: P \longrightarrow M$
satisfying $\varphi\beta = \alpha$. 
\medskip
\begin{center}
\begin{tikzcd}
& M \arrow[d, two heads, "\varphi"] \\
P \arrow[ur, dashed, "\beta"]  \arrow[r, "\alpha"] & N
\end{tikzcd}
\end{center}
\medskip
It is well-known that in any variety of algebras the projective objects are the retracts of free objects. 
In the category of semimodules over a semiring $A$, the free object over a set $X$ is $A^{(X)}=\{f:X\to A\mid f(x)=0\text{ for all but finitely many $x\in X$}\}$ (\cite{dr:sasiimva}). So, we obtain the following characterization of projective semimodules.

\begin{thm}
Let $A$ be a semiring. An $A$-semimodule $P$ is projective if and only if it is a retract of the semimodule $A^{(X)}$ for some set $X$.
\end{thm}

\section{Injective and projective semimodules over involutive semirings}

Let $A=(A, \vee, \cdot, 0, 1)$ be an additively idempotent semiring and $(M, \vee, 0)$ a semimodule over it. 
$M$ is called \emph{MID-complete} if the semilattice $M$ is a complete lattice, and if $M$ satisfies the meet infinite distributive identity (MID), i. e. 
$m\vee \bigwedge_{i \in I} m_{i} = \bigwedge_{i \in I} (m \vee m_{i})$.

A join-semilattice $(A, \vee, 0)$ is join-distributive if for any $a$, $b_0$ and $b_1$ elements of $A$ such that $a \leq b_0 \vee b_1$, then there exist $a_0$, $a_1$ $ \in A$ such that $a_0 \leq b_0$, $a_1 \leq b_1$ and $a = a_0 \vee a_1$.

An \emph{ideal} of a join-semilattice $(A, \vee, 0)$ is a subset $I$ of $A$ such that
\begin{itemize}
\item if $a, b \in I$, then $a \vee b \in I$;
\item if $ a \in I$, $b \in A$ and $b \leq a$, then $b \in I$.
\end{itemize}

In the rest of the section we shall write ideal of a semiring $A$ meaning that such ideal is to be considered as an ideal according to the above definition (i.e. ideal of a join-semilattice and not ideal of a semiring).

The following result is well known from lattice theory.
\begin{lem}
Let $A$ be a join-semilattice. Then the lattice of ideals $(\Id(A), \cap, \wedge)$, ordered by reverse inclusion, is complete. If $A$ is join-distributive, then $\Id(A)$ is MID-complete
(i.\,e. $J \cap \bigwedge_{i \in I} J_{i \in I} = \bigwedge_{i \in I} (J \cap J_{i \in I})$, for any $J, J_{i} \in \Id(A)$).
\end{lem}
\begin{proof}
For the completeness it is sufficient to observe that for any $J_{i} \in \Id(A)$ we have that $\bigvee_{i \in I} J_{i} = \bigcap_{i \in I} J_{i}$ and since the set of ideals of a lattice is closed under arbitrary intersections we have that the lattice is complete. For the second part, observe that $\bigwedge_{i \in I} J_{i} = \{ \vee_{k=1}^n a_{i_k} \mid a_{i_k} \in J_{i_k}, \{i_1, \dots, i_n\} \subseteq I, n \in \mathbb{N} \}$.
This set is obviously closed under joins, to see that it is downward closed consider an element $x \leq a \in \bigwedge_{i \in I} J_{i}$, then $a =  a_{i_{1}} \vee \dots \vee a_{i_{n}}$ where $a_{i_{k}} \in J_{i_{k}}$, for some $J_{i_{k}} \in \Id(A)$ for every $k=1,\dots,n$. 
Then, since $A$ is join-distributive, we have that there exist elements $a_{i_{1}}'$, $\dots$, $a_{i_{1}}'$ such that $a_{i_{k}}' \leq a_{i_{k}}$ for every $k \in \{1, \dots , n\}$ and $x =  a_{i_{1}}' \vee \dots \vee a_{i_{n}}'$. Since any ideal  $J_{i_{k}}$ is downard closed we have that $a_{i_{k}}' \in J_{i_{k}}$ for every $k \in \{1, \dots , n\}$, so $x \in \bigwedge_{i \in I} J_{i}$. 
It is now straightforward to see that $J \cap (\bigwedge_{i \in I} J_{i}) = \bigwedge_{i \in I} (J \cap J_{i})$ and the proof is complete. 
\end{proof}

For an idempotent semiring $A$, an element $a\in A$ and an ideal $I\subseteq A$, define scalar multiplication by $a\cdot I=\{x\in A\mid xa\in I\}$. 
Then $a\cdot I$ is also an ideal of $A$, and 
it is straight forward to check that $(\Id(A),\cap,I)$ is an $A$-semimodule (ordered by reverse inclusion). Recall also that for a semiring
homomorphism $f:A\to B$, $\Ker(f)=\{x\in A\mid f(x)=0\}$, and this is a member of $\Id(A)$.

\begin{thm}\label{homid}
Let $A$ be an additively idempotent semiring. Then $\Hom_{\mathbb{B}}(A, \mathbb{B})$ and $\Id(A)$ are isomorphic as $A$-semimodules.
\end{thm}

\begin{proof} As noted above, $\Ker$ is a map from $\Hom_{\mathbb{B}}(A, \mathbb{B})$ to $\Id(A)$,
and since a function $f:A\to\mathbb B$ is determined by the preimage of $\{0\}$, the map $\Ker$ is a bijection. For $f,g\in \Hom_{\mathbb{B}}(A, \mathbb{B})$ and $a\in A$ we have 
$$\Ker(f\vee g)=\{x\in A\mid (f\vee g)(x)=0\}=\Ker(f)\cap \Ker(g) \text{ and}$$
$$\Ker(a\cdot f)=\{x\in A\mid (a\cdot f)(x)=0\}=\{x\in A\mid f(xa)=0\},$$
which agrees with $a\cdot \Ker(f)=\{x\in A\mid xa\in \Ker(f)\}$.
\end{proof}

With this result we can restate Theorem~\ref{inj}.

\begin{cor}\label{inject}
Let $A$ be an additively idempotent semiring and $M$ an $A$-semimodule. Then, $M$ is injective if and only if $M$ is a retract of $\Id(A)^{X}$ for some set $X$.
\end{cor}

\begin{lem}
Let $(B, \vee, 0)$ and $(M, \vee, 0)$ be two semimodules over an idempotent semiring $A$ and suppose that $M$ is a retract of $B$. If $B$ is MID-complete
then $M$ is also MID-complete.
\end{lem}

\begin{proof}

Let $\alpha: M \to B$ and $\beta : B \to M$ be the two homomorphisms which determine the retraction. 
We prove that $M$ is a complete semimodule and that $\bigwedge_{i\in I}m_i = \beta(\bigwedge_{i \in I} \alpha(m_{i}))$. 
Indeed, we first note that $$m_i = \beta\alpha(m_i)\geq \beta\alpha(\bigwedge_{i\in I}m_i)$$ for all $i\in I$. If $m'\in M$ and $m_i\geq m'$ for all $i\in I$, then we have that $\alpha(m_i)\geq \alpha(m')$ for all $i\in I$, and hence $\bigwedge_{i\in I}\alpha(m_i)\geq \alpha(m')$. This implies that $\beta(\bigwedge_{i\in I}\alpha(m_i))\geq \beta(\alpha(m')) = m'$. Therefore, $\bigwedge_{i\in I}m_i$ exists in $M$ and is equal to $\beta(\bigwedge_{i\in I}\alpha(m_i))$. So, $M$ is a complete.

Since $B$ satisfies the MID law, we have
\begin{equation*}
\begin{array}{rcl}
m \vee \bigwedge_{i\in I}m_i &=& \beta(\alpha(m)) \vee \beta(\bigwedge_{i\in I}\alpha(m_i))= \beta(\alpha(m) \vee \bigwedge_{i\in I}\alpha(m_i))\\
&=& \beta(\bigwedge_{i\in I}(\alpha(m)\vee\alpha(m_i)))= \beta(\bigwedge_{i\in I}\alpha(m\vee m_i))\\
&=&\bigwedge_{i\in I}(m\vee m_i),
\end{array}
\end{equation*}
so, $M$ is MID-complete and the statement is proved.
\end{proof}

\begin{thm} 
\label{Theorem 9}
Let $A$ be a distributive and additively idempotent semiring and $M$ an injective semimodule over $A$. Then $M$ is MID-complete.
\end{thm}

\begin{proof}
We know that $M$ is injective if and only if it is a retract of $\Hom_{\mathbb{B}}(A, \mathbb{B})^{X}$ for some set $X$. Since we know that $\Hom_{\mathbb{B}}(A, \mathbb{B})$ and $\Id(A)$ are isomorphic as $A$-semimodules (and obviously as join-semilattices) we have that $\Hom_{\mathbb{B}}(A, \mathbb{B})$ is complete and infinitely distributive and so $\Hom_{\mathbb{B}}(A, \mathbb{B})^{X}$. From the previous theorem we obtain that $M$ is MID-complete. 
\end{proof}

Recall that in a pointed residuated join-semilattice we define $-x=0/x$ and ${\sim}x=x\backslash 0$.

\begin{thm}
Let $A$ be a finite $1$-bounded pointed residuated join-semilattice.
Then $A$ is an involutive semiring if and only if $A$ and $\Id(A)$
are isomorphic as $A$-semimodules via the map $\Phi(a)={\downarrow}{-}a$.
\end{thm}

\begin{proof}
By Theorem~\ref{homid} we can consider $\Id(A)$ in place of $\Hom_\mathbb B(A,\mathbb B)$.
First, assume $A$ is a finite $1$-bounded involutive semiring and define a map $\Phi : A \to \Id(A)$ by $\Phi (a) = {\downarrow}{{-}a}=\{x\in A\mid x\le {-}a\}$, where $-a=0/a$. Since every ideal of a finite join-semilattice is principal, and since $-$ is a bijection, this map is also bijective. It is order-preserving since $-$ is order-reversing and $\Id(A)$ is ordered by reverse inclusion, hence 
$\Phi(a\vee b)=\Phi(a)\cap\Phi(b)$. The following calculation shows that $\Phi$ preserves scalar multiplication:
$$b\cdot \Phi(a)=\{x\in A\mid xb\le-a\}=\{x\in A\mid xba\le 0\}=\{x\in A\mid x\le-(ba)\}=\Phi(ba).$$

Conversely, assume $A$ is a finite $1$-bounded residuated join-semilattice, and $A$, $\Id(A)$ are isomorphic as $A$-semimodules via the map $\Phi(a)={\downarrow}{-}a$, where $-a=0/a$ and 
$0$ is the bottom element of $A$. Let $f(a)=\bigvee\Phi(a) = -a$. Since $A$ and $\Id(A)$ are assumed to be isomorphic, $f$ is a bijection. From residuation it follows that $x\le 0/y\iff xy\le 0\iff y\le x\backslash 0$, hence $-,\sim$ form a Galois connection, hence ${-}{\sim}$ and ${\sim}{-}$ are closure operators and ${-}{\sim}{-}x=-x$. Since $f(x)=-x$ is a bijection, we get ${\sim}{-}x=x$
and ${-}{\sim}x=x$, so $A$ is an involutive semiring by Theorem~\ref{termeq}.
\end{proof}

The previous theorem together with Corollary~\ref{inject} gives the following result.

\begin{cor}
Let $A$ be a finite $1$-bounded involutive semiring and $M$ a semimodule over $A$. Then $M$ is injective if and only if it is a retract of $A^X$ for some set $X$.
\end{cor}

\begin{thm}\label{injiffproj}
Let $A$ be a finite $1$-bounded involutive semiring and $M$ a finitely generated $A$-semimodule. Then, $M$ is injective if and only if it is projective.
\end{thm}

\begin{proof}
Since $A \cong \Hom_{\mathbb{B}}(A, \mathbb{B})$ as $A$-semimodules, we have that retracts of $A^{X}$ for some finite set $X$ (projective semimodules) are exactly the retracts of $\Hom_{\mathbb{B}}(A, \mathbb{B})^{X}$ (injective semimodules).
\end{proof}

We can wonder in which cases injective and projective semimodules coincide and in particular if we can weaken the hypothesis about the semiring assumed in the above theorem. As regards involution the answer is no and we shall provide an example. 

\begin{example}
Consider the three-elements idempotent semiring $A = \{0, a, 1\}$ with $0 < a < 1$ and $a \cdot a = a$, then injective and projective semimodules over this semirings don't coincide. First of all observe that $Id(A) = \{0, {\downarrow}a, A\}$. We know that $A$ is a projective semimodule over itself. We shall now prove that $A$ can't be injective. Suppose that $A$ is self-injective, so it should be a retract of $Id(A)^n$ for some finite $n \in \mathbb{N}$ since $A$ is finitely generated. In this case, we should have an $A$- semimodule morphism $\Phi: Id(A)^n \to A$ such that $Im(\Phi) = A$, if $\Phi( \{0\}, \{0\}, \dots, \{0\})$ is $a$ or $0$, then $|Im(\Phi)|\leq 2$ ($\Phi$ is order-preserving), in particular $Im(\Phi) \neq A$, so $\Phi( \{0\}, \{0\}, \dots, \{0\}) = 1$, but in this case
\[
1 = \Phi( \{0\}, \dots, \{0\}) = \Phi( a \cdot(\{0\}, \dots, \{0\})) = a \cdot \Phi( \{0\}, \dots, \{0\}) = a \cdot 1 = a
\]
which is absurd.
\end{example}

\begin{thm}
Let $A$ be a finite $1$-bounded involutive semiring and $M$ a cyclic $A$-semimodule. Then the following are equivalent:
\begin{enumerate}
\item $M \cong Au$ for some $u \in A$ multiplicatively idempotent (i.\,e. $u \cdot u = u$);
\item $M$ is projective;
\item $M$ is injective.
\end{enumerate}
\end{thm}

\begin{proof}
The equivalence between (1) and (2) is true for any semiring (see \cite[Remark 3.4]{dr:sasiimva}). The equivalence between (2) and (3) is proved in the previous theorem.
\end{proof}

\begin{lem} \textup{(}\cite{dlnv:isaidscmv}\textup{)}
Let $A= \prod_{i\in I}A_i$ be a direct product of semirings $A_i$. Then $A$ is self-injective if and only if each $A_i$ is self-injective.
\end{lem}

\begin{cor}
Every direct product of finite involutive semirings is self-injective.
\end{cor}

\begin{proof}
Let $A$ be a finite commutative involutive semiring. It is clear that $A$ is self-injective since it is isomorphic to $\Hom_{\mathbb{B}}(A, \mathbb{B})$ and the corollary is proved.
\end{proof}

\section{Strong semimodules and semimodules over n-potent involutive semirings}

A semimodule $M$ over an involutive semiring $A$ is \emph{strong} if for all $a,b\in A$
$$
\forall m\in M\ (a \cdot m = b \cdot m)\implies \forall m\in M\ (-a \cdot m = -b \cdot m\text{ and }{\sim}a \cdot m = {\sim}b \cdot m).
$$

A semiring $A$ is called $nilpotent$ if for every $a \in A$, $a \neq 1$, there exists a $n \in \mathbb{N}$ such that $a^{n}=0$.
An $A$-semimodule $M$ is \emph{faithful} if the action of each $a \neq 0$ in $A$ on $M$ is nontrivial, i.\,e. $a \cdot x \neq 0$ for some $x \in M$.

\begin{thm}
Let $A$ be a nilpotent $1$-bounded involutive semiring and $M$ a nontrivial $A$-semimodule. Then $M$ is a strong semimodule if and only if $M$ is faithful.
\end{thm}

\begin{proof}
Note that for any $a \in A$ we have that $a \cdot ({\sim}a) = (-a) \cdot a = 0$. 
 Suppose $M$ is faithful and let $a \cdot x= b \cdot x$, for all $x \in M$. Then we have $0=((-a)a) \cdot x=((-a)b) \cdot x$ for all $x \in M$ and also $((-b)a) \cdot x=0$ for all $x \in M$. Since $M$ is faithful we have $(-a)b=(-b)a=0$, which imply respectively that $b\le a$ and $a \le b$. Consequently we have $a=b$ and obviously $-a \cdot x= -b \cdot x$ and ${\sim}a \cdot x = {\sim}b \cdot x$ for all $x \in M$.
Vice versa, suppose $M$ is strong and that $a \cdot x=0=0 \cdot x$ for all $x \in M$, for some $0 \neq a$ in $A$. Then we have $-a \cdot x=-0 \cdot x= 1 \cdot x=x$ for all $x \in M$, which implies $(-a)^{n} \cdot x=x$ for all $x \in M$ and $n \in \mathbb{N}$. But, since $A$ is nilpotent, we have that $-a=1$ and so $a=0$, which contradicts the hypothesis.
\end{proof}

A semiring $A$ is \textit{multiplicatively idempotent} if $x \cdot x = x$ for every $x \in A$.

\begin{thm}
A $1$-bounded involutive semiring $A$ is multiplicatively idempotent if and only if $A$ is a Boolean algebra.
\end{thm}

\begin{proof}
From $xx=x\le 1$ it follows that $x \cdot y = x \wedge y$, so the semiring is commutative and, in particular, $-x={\sim} x$. Defining $x \rightarrow y$ as ${\sim}((-y) \cdot x) = - ((-y) \cdot x)$, we obtain that $(A, \vee, \wedge, \rightarrow, 0, 1)$ is a Heyting algebra. We have $\neg x$ defined as $x \rightarrow 0$ and so $\neg \neg x = (x \rightarrow 0) \rightarrow 0 = - (1 \cdot (-x)) -(-x)=x$. Therefore the Heyting algebra is a Boolean algebra. 
\end{proof}

Let $A$ be a $1$-bounded involutive semiring. 
For a $n \in \mathbb{N}$, a semiring $A$ is $n$-\emph{von Neumann regular} if for every $a \in A$, there exists $b \in A$ such that $a^n = a^n \cdot b \cdot a^n$. A $1$-von Neumann regular semiring is simply called \emph{von Neumann regular}.
A semiring $A$ is $n$-\emph{potent} if $a^n = a^{n+1}$ for every $a \in A$, in particular a semiring $A$ is 1-potent if and only if it is multiplicatively idempotent. 

\begin{thm}
Let $A$ be a 1-bounded idempotent semiring and $n\in\mathbb N$. Then $A$ is $n$-von Neumann regular if and only if $A$ is $n$-potent.
\end{thm}

\begin{proof} We first prove the result for $n=1$.

$(\Rightarrow)$ Note that, since $A$ is 1-bounded, we have that $a \cdot a \leq a$, for any $a \in A$. Suppose that $A$ is von Neumann regular, so $a \vee (a \cdot a) = (a \cdot b \cdot a) \vee (a \cdot 1 \cdot a) = a \cdot (b \vee 1) \cdot a = a \cdot a$. Since $a \cdot a \leq a$, we have that $a \cdot a = a$.

$(\Leftarrow)$ Suppose $A$ is multiplicatively idempotent, then $a = a \cdot a = a \cdot 1 \cdot a$.

Now let $n\in\mathbb N$. From the two implications proved above,
we have that $A$ is $n$-von Neumann regular iff $a^{2n}=a^{n}$. Since $a^{2n} \leq a^{n+1} \leq a^{n}$, this implies $a^{n+1}=a^{n}$. Obviously $a^{n}=a^{n+1}$ implies $a^{2n}=a^{n}$ for any $a \in A$.
\end{proof}

\begin{thm}
For a $1$-bounded involutive semiring $A$, the following statements are equivalent:
\begin{enumerate}
\item Every left principal semiring ideal $A a$ of $A$ is injective as a semimodule;
	
\item $A$ is a self-injective von Neumann regular semiring;
	
\item $A$ is a complete Boolean algebra.	
\end{enumerate}
\end{thm}
\begin{proof}
(1)$\Rightarrow$(2). We need only to show that $A$ is a von Neumann regular semiring. Indeed, let $a \in A$. Then, by condition (1), $A \cdot a$ is an injective $A$-semimodule. We then have that there exists an $A$-homomorphism $f: A\longrightarrow A \cdot a$ such that $f|_{A \cdot a} = id_{A \cdot a}$. It implies that $$a = f(a) = f(a\cdot 1) = af(1).$$ On the other hand, since $f(1) \in A \cdot a$, there exists an element $b\in A$ such that $f(1) = b\cdot a$, and hence, $a = a\cdot b\cdot a$. Thus, $A$ is von Neumann regular.

(2)$\Rightarrow$(3). Since $A$ is a self-injective semiring and applying Theorem \ref{Theorem 9}, the lattice $A$ is complete.

From the previous theorems we know that, since $A$ is von Neumann regular then it is idempotent and consequently a Boolean algebra.

(3)$\Rightarrow$(1). Suppose $A$ is a complete Boolean algebra. By \cite[Corollary 2]{fofa:iopoba}, $A$ is a self-injective semiring. Take any $a\in A$. We have that $a\cdot a = a$. Define two $A$-homomorphisms $\alpha: A \cdot a\longrightarrow A$ and $\beta: A\longrightarrow A \cdot a$ by setting: $\alpha(b\cdot a) = b\cdot a$ and $\beta(b) = b\cdot a$ for all $b\in A$. It is obvious that $\beta\alpha = id_{A \cdot a}$; that means, $A \cdot a$ is a retract of the $A$-semimodule $A$. By $A$ is self-injective and by \cite[Lemma 3.1]{dlnv:isaidscmv}, $A \cdot a$ is an injective $A$-semimodule, and hence, statement (1) is proved, finishing the proof.
\end{proof} 

\begin{thm}
Let $A$ be a 1-bounded semiring. Then, for a fixed $n \in \mathbb{N}$, the following statements are equivalent:
\begin{enumerate}
\item for every $a \in A$ the cyclic semimodule generated by $a^n$ is injective as a semimodule on $A$;
\item $A$ is self-injective and n-potent.
\end{enumerate}
\end{thm}

\begin{proof}
$(1) \Rightarrow (2)$ Obviously $A$ is self-injective since it is generated by $1^n$. If $A \cdot a^n$ is injective, then exists a $A$ - homomorphism $f: A \to A \cdot a^n$ such that $f|_{A \cdot a^n} = id_{A \cdot a^n}$. It implies that $a^n=f(a^n)=f(a^n \cdot 1)= a^n \cdot f(1)$. Since $f(1) \in A \cdot a^n$, we have that exists an element $b \in A$ such that $f(1)=b \cdot a^n$, so $a^n= a^n \cdot b \cdot a^n$.

We then get that $A$ is n-von Neumann regular and for a previous remark $a^n = a^{n+1}$, for every $a \in A$.\\

$(2) \Rightarrow (1)$ Define $\alpha: A \cdot a^n \to A$ by $\alpha (b \cdot a^n) = b \cdot a^n$ and $ \beta : A \to A \cdot a^n$ by $ \beta (b)= b \cdot a^n$. We then have that $ \beta \alpha (b \cdot a^n) = b \cdot a^{2n}$. Since $a^n=a^{n+1}$ implies $a^{2n}=a^n$ and consequently $b \cdot a^{2n} = b \cdot a^n $, we have that $ \beta \alpha = id_{A \cdot a^n}$. So, $A \cdot a^n$ is a retract of $A$ which is self-injective. This implies that $A \cdot a^n$ is injective too.
\end{proof}

As an example, consider the finite commutative involutive linearly-ordered $1$-bounded involutive semiring $C=\{0, a, b, 1\}$ where $0<a<b<1$, $a \cdot a=0$ and $b \cdot b = b$.
\medskip

\begin{center}
\begin{tikzpicture}[scale=0.7]
\node (1) at (0,2) {$1$};
\node (b) at (0,1) {$b$};
\node (b') at (2,1) {$b^2=b$};
\node (a) at (0,0) {$a$};
\node (a') at (2,0) {$a^2=0$};
\node (0) at (0,-1) {$0$};
\draw (0)--(a)--(b)--(1);
\end{tikzpicture}
\end{center}
\medskip

It is easy to see that $C$ is 2-potent and 
we know that $C$ self-injective since it is a projective $C$-semimodule and therefore injective
by Theorem~\ref{injiffproj}.
Hence all the cyclic semimodules of the form $Cc^{n}$ for some $c \in C$ are injective and projective. In particular we have that the semimodules 
$\{0\}$, $Cb$ and $C$ are injective and, using Theorem~\ref{injiffproj} again, also projective.

\noindent\textbf{Acknowledgement}. The second author is very grateful for support from the SYSMICS project and for the opportunity to spend a month doing research at Chapman University.

\end{document}